\documentclass[a4paper,leqno,11pt]{article}

\usepackage{amsmath,amssymb,amsthm,
verbatim
}

\usepackage{tikz}
\usetikzlibrary{decorations.pathreplacing}

\newcommand\be{\begin{equation}}
\newcommand\ee{\end{equation}}
\newcommand\R{{\mathbb R}}

\newcommand\implica{\quad\Rightarrow\quad}

\newcommand\G{\mathcal G}

\newcommand\vv{\textsc{v}}
\newcommand\ww{\textsc{w}}

\newcommand\mis{\mathop{\rm meas}}

\newcommand\eps{\varepsilon}

\newtheorem{theorem}{Theorem}[section]
\newtheorem{proposition}[theorem]{Proposition}
\newtheorem{lemma}[theorem]{Lemma}

\theoremstyle{remark}
\newtheorem{remark}[theorem]{Remark}
\newtheorem*{remark*}{Remark}

\theoremstyle{definition}

\newtheorem{example}[theorem]{Example}

\tikzstyle{nodo}=[circle,draw,fill,inner sep=0pt,minimum size=\widthof{k}]
\tikzstyle{infinito}=[circle,inner sep=0pt,minimum size=0mm]

\date{}

\newcommand\iii{$\infty$}

\newcommand\dx{{\,dx}}
\newcommand\dt{{\,dt}}

\title{NLS ground states on graphs}

\author{Riccardo Adami\thanks{Author partially supported by the FIRB 2012 project ``Dispersive dynamics: Fourier
Analysis and Variational Methods".}, Enrico Serra\thanks{Author
partially supported by the PRIN 2012 project ``Aspetti
variazionali e perturbativi nei problemi differenziali
nonlineari".}, Paolo Tilli \\ \ \\{\small  Dipartimento di Scienze
Matematiche ``G.L. Lagrange'', Politecnico di Torino } \\ {\small
Corso Duca degli Abruzzi, 24, 10129 Torino, Italy}}

\begin{document}

\maketitle

\begin{abstract}
We investigate the existence of ground states for the
subcritical NLS energy on metric graphs. In particular, we find out a
topological assumption that guarantees the nonexistence of ground
states, and give an example in which the  assumption is not fulfilled
and ground states actually exist.
In order to obtain the result, we introduce a new rearrangement
technique, adapted to the graph where it applies. Owing to such a
technique, the energy level of the rearranged function is improved by
conveniently mixing the symmetric and monotone rearrangement procedures.
\end{abstract}

\noindent{\small AMS Subject Classification: 35R02, 35Q55, 81Q35, 49J40, 58E30.}
\smallskip

\noindent{\small Keywords: Minimization, metric graphs, rearrangement, nonlinear Schr\"odinger \\ \hbox{} \hskip 1.65cm Equation.}

\section{Introduction}

In this paper we investigate the existence of a ground state for
the NLS energy functional
\begin{equation}
\label{NLSe}
E (u,\G)
  =  \frac 1 2 \| u' \|^2_{L^2 (\G)}
- \frac 1 p  \| u \|^p_{L^p (\G)}
\end{equation}
 with the {\em mass constraint}
\begin{equation}
\label{mass}
\| u \|^2_{L^2 (\G)} \ = \ \mu,
\end{equation}
where
 $\mu>0$ and $p\in (2,6)$ are given numbers
 and $\G$ is a \emph{connected metric graph}.

Here we present a rather informal description of the
problem and of the main results of the paper, whereas
a precise setting
and formal definitions are given in Section~\ref{sec:defi},

A metric graph $\G$  (\cite{berkolaiko, Friedlander,Kuchment})
is essentially a one-dimensional singular variety,
made up of several, possibly unbounded intervals (the edges of the
graph) some of whose endpoints
are glued together according to the topology of the graph.
The spaces
$L^p(\G)$, $H^1(\G)$ etc. are defined in the natural way. All
the functions we consider
are real valued: this is not restrictive, because
$E(|u|,\G)\leq E(u,\G)$ and any ground state is in fact real valued,
up to multiplication by a constant phase $e^{i\theta}$.

When $\G=\R$  the minimization problem
\begin{equation}
\label{minE}
\min E(u,\G),\quad
u\in H^1(\G),\quad
\| u\|_{L^2(\G)}^2=\mu
\end{equation}
is well understood and the minimizers, called \emph{solitons}, are known
explicitly. The same is true when $\G=[0,+\infty)$ is a half-line (the
minimizer being ``half a soliton'' of mass $2\mu$) and, to some extent,
when $\G=[a,b]$ is a bounded interval.
Much more interesting is the case when
$\G$ is non-compact and has a nontrivial topology with multiple junctions, loops and so on
(see Figure~\ref{Fig_gen}).
The aim of this paper is that of studying
existence and  qualitative properties of solutions
to \eqref{minE},  under quite
general assumptions on $\G$ (for papers devoted to particular graphs, see \cite{acfn12a, acfn14,cacciapuoti}).

Since when $\G$ is compact existence of minimizers for \eqref{minE} is
immediate, we focus on graphs where at least one edge is unbounded (a half-line),
so that the embeddings $H^1(\G)\hookrightarrow L^r(\G)$ are not compact and
existence for \eqref{minE} is non-trivial. In fact,
even though the infimum of $E(u,\G)$ is always trapped between two
finite values (Theorem~\ref{teoremzero}),
it turns
out that the existence of minimizers heavily depends on the topology of
$\G$: if, for instance, $\G$ consists of two half-lines with a ``double bridge" in between
(Figure~\ref{Figbridge}.a) then \eqref{minE} has \emph{no solution}, while if $\G$
is a straight line with one pendant attached to it (Figure~\ref{Figbridge}.b) then minimizers
\emph{do exist}.

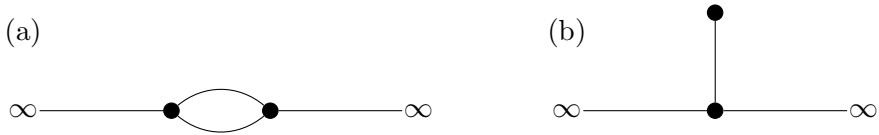
\begin{figure}[t]
\begin{center}
\begin{tikzpicture}
[scale=1.3,style={circle,
              inner sep=0pt,minimum size=7mm}]
%
%
\node at (-1.5,0) [infinito]  (1) {$\infty$};
\node at (0,0)[nodo] (2) {};
\node at (1,0) [nodo] (3) {};
\node at (2.5,0) [infinito] (4) {$\infty$};
\draw [-] (2) -- (1) ;
\draw [-] (4) -- (3) ;
\draw [-] (2) to [out=-40,in=-140] (3);
\draw [-] (2) to [out=40,in=140] (3);
\draw (-1.5,0.8) node {(a)};
  %
  %
\node at (4,0) [infinito]  (L) {$\infty$};
\node at (5.5,0)[nodo] (O) {};
\node at (7,0) [infinito] (R) {$\infty$};
\node at (5.5,1) [nodo] (V) {};
\draw [-] (L) -- (O) ;
\draw [-] (O) -- (R) ;
\draw [-] (O) -- (V) ;
\draw (4,0.8) node {(b)};
\end{tikzpicture}
\end{center}
\caption{(a) Two half-lines with a double-bridge in between.
(b) A straight line with one pendant attached to it.}
\label{Figbridge}
\end{figure}

Our results extend in two directions. First, we prove
a nonexistence result (Theorems~\ref{teoremunoA} and \ref{teoremunoB}) for a broad family of
non-compact graphs (see condition (H) in Section~\ref{sec:defi}): roughly speaking, if
no cut-edge of $\G$ segregates all the half-lines of $\G$ in the same connected component,
then \eqref{minE} has no solution, the only exceptions being certain graphs with a
particular topology which we characterize completely (see Example~\ref{example1}).
Moreover, for all these graphs
the unattained infimum of $E(u,\G)$ coincides with the NLS energy of a soliton
on the real line, for the same values of $\mu$ and $p$.

The mentioned condition on $\G$, that prevents existence
of minimizers in \eqref{minE}, is incompatible with the presence of a
bounded pendant edge
attached to $\G$.
This motivates
 a case study when $\G$ is the simplest non-compact graph with  one
 pendant,
namely
the graph in Figure~\ref{Figbridge}.b:
we prove that for this particular $\G$
problem \eqref{minE} does have a solution (Theorem~\ref{teoremdue}) and we establish
some qualitative properties of the minimizers (Theorem~\ref{teoremdueB}). In fact, in this case
the energy level of the minimizers is \emph{strictly lower} than the energy level of a soliton
on the real line, and any minimizer $u$ ---in order to reduce the energy level---
does exploit the topology of $\G$, in that $\sup u$ is attained at the tip of the
pendant: indeed, the fact that concentrating mass on the pendant is energetically convenient
prevents an a priori possible loss of compactness
of a minimizing sequence along a half-line, and existence can be proved.
In this respect, a key role
 is played
by a new rearrangement technique, introduced in the proof of Lemma~\ref{lemmaH}.
The interesting feature is that a \emph{hybrid} rearrangement of
a function
$u\in H^1(\G)$ is needed, adapted
to the topology of $\G$: \emph{high} values of $u$ are rearranged \emph{increasingly}
on the pendant, while \emph{small} values of $u$ are rearranged \emph{symmetrically}
on the straight line of $\G$.

It should be pointed out, however, that the presence of a bounded pendant attached
to $\G$ is not enough, alone, to guarantee solutions to \eqref{minE}.
The problem of characterizing all non-compact $\G$ such that \eqref{minE} has
a solution is certainly a challenging one, since the topology of
$\G$ alone is not enough to answer this question and, in general, also the \emph{metric}
properties of $\G$ (i.e. the lengths of it edges) play a relevant role. This issue
will be discussed in more detail in a forthcoming paper.

\bigskip

Among the physical motivations for this problem (see
\cite{noja} and references therein), nowadays the most topical is probably
given by the Bose-Einstein condensation (see \cite{stringari}).
It is  widely known that, under a critical
temperature, a boson gas undergoes a phase transition that leads a
large number $N$ of
particles into the same quantum state, represented by the {\em wave
  function}
 that minimizes
the {\em Gross-Pitaevskii energy} functional
\be \label{gp}
E_{GP} (\varphi, \Omega) \ = \ \| \varphi' \|^2_{L^2 (\Omega)}
+ 8 \pi \alpha  \| \varphi \|^4_{L^4 (\Omega)}
\ee
under the  the {\em normalization condition}
$
\| \varphi \|_{L^2 (\Omega)}^2 \ = \ N.
$
The real number
$\alpha$  is the scattering length associated
with the two-body
interaction between the particles in the gas. The functional
in \eqref{NLSe} corresponds to the case of a negative scattering length,
that is realized, for instance, by an attractive two-body interaction.
Besides, in \eqref{NLSe} we consider a general subcritical nonlinearity power.

\noindent
In \eqref{gp}, the domain $\Omega$
corresponds
to the shape of the (magnetic and/or optical) trap where the
gas has been confined in order to induce the phase transition. Present
technology allows various
shapes for traps, like discs, cigars, and so on. Recently, the possibility of
building ramified traps (\cite{kevrekidis,vidal}) has been
envisaged theoretically,
even though, at least to
our knowledge, they have not been experimentally realized
so far.

\noindent
To give a mathematical description
of such an experimental setting, one should choose a spatial domain
$\Omega$ that reproduces the shape of the trap and then minimize the energy
in \eqref{gp}. One would expect that, for branched traps, the domain
 $\Omega$ may be
replaced by a suitable graph $\G$. The possible ground state then
provides the state of the condensate in the trap $\Omega$, while the
absence of a ground state would in principle signal an instable
character of the system. For instance, in the situation depicted by
Theorems \ref{teoremzero}, \ref{teoremunoA}, the system would run away
along an infinite edge, mimicking the shape of a soliton.
For further results on nonlinear evolution on graphs, see e.g. \cite{ali, bona, smi,matrasulov}.

The paper is organized as follows: Section \ref{sec:defi}
contains a precise setting of the problem and
the
statements of the main results.
In Sections \ref{sec:prelim} and \ref{sec:aux} we discuss some
preliminary facts and techniques (in particular, rearrangements on graphs)
and some auxiliary statements that may have some interest in themselves.
Finally, Sections \ref{sec:proofs1} and \ref{sec:baffo} contain the proofs
of the results stated in Section~\ref{sec:defi}.

\emph{Remark on Figures:} some figures have been included to better describe the
topology of certain metric graphs. In these pictures, vertices ``at infinity'' are denoted
by the symbol $\infty$, while ordinary vertices (i.e. junctions of two or more
edges) are denoted by a bullet.
\section{Setting, notation and main results} \label{sec:defi}

Although we shall not need deep results from graph theory,
the notion of graph is central to this paper: we refer the reader to
\cite{balakrishnan,bollobas} for a modern account on the subject.

Throughout the paper a graph is always
meant as a (connected) \emph{multigraph},
that is, we allow for multiple edges joining the same pair of vertices.
Self-loops (i.e. edges starting and ending at the same vertex) are also
allowed. More precisely,
the central objects of the paper
are \emph{metric graphs} (see \cite{Kuchment,
Friedlander}), i.e. (connected) graphs $\G=(V,E)$ where
each edge $e\in E$ is
associated with either
a closed bounded interval $I_e=[0,\ell_e]$ of length $\ell_e>0$,
or a closed half-line $I_e=[0,+\infty)$, letting $\ell_e=+\infty$
in this case. Two edges $e,f\in E$ joining the same pair of vertices, if present,
are \emph{distinct objects} in all respects: in particular, the corresponding
intervals $I_e$ and $I_f$ need not have the same length, and must be considered
\emph{distinct} even in the case where $\ell_e=\ell_f$.
For every $e\in E$ joining two vertices $\vv_1,\vv_2\in V$,
a coordinate $x_e$ is chosen along $I_e$, in such a way that $\vv_1$
corresponds to $x_e=0$ and $\vv_2$ to $x_e=\ell_e$, or \emph{viceversa}:
if $\ell_e=+\infty$, however, we always assume that the half-line $I_e$
is attached to the remaining part of the graph at $x_e=0$, and the vertex
of the graph corresponding to $x_e=+\infty$ is called a \emph{vertex at infinity}.
The subset of $V$ consisting of all vertices at infinity will be denoted by
$V_\infty$.
With this respect, we shall always assume that
\begin{equation}
\label{H0}
\text{all vertices at infinity of $\G$ (if any) have degree one}
\end{equation}
where, as usual, the degree
 of a vertex $\vv\in V$ is the
number of edges incident at $\vv$ (counting twice any self-loop at $\vv$, of course).
Finally, the cardinalities of $E$ and $V$
are assumed to be finite. An example of a typical metric graph $\G$ is given in
Figure~\ref{Fig_gen}.

\medskip

\noindent A connected metric graph $\G$ has the natural structure of a locally compact
metric space, the metric being given by the shortest distance measured along the
edges of the graph. Observe that
\begin{equation}
\label{comp}
\text{$\G$ is compact $\iff$ no edge of $\G$ is a half-line $\iff$ $V_\infty=\emptyset$.}
\end{equation}
With some abuse of notation, we often identify an edge $e$ with the corresponding
interval $I_e$:
thus, topologically, the \emph{metric space}
$\G$ is the \emph{disjoint} union $\bigsqcup I_e$
of its edges, with some of their endpoints \emph{glued together} into a single point
(corresponding to a vertex $\vv\in V\setminus V_\infty$),
according to the topology of the \emph{graph} $\G$ (using the same
symbol $\G$ for both the metric graph and the induced metric space
should cause no confusion).
We point out  that
any vertex at infinity $\vv\in V_\infty$ is of course a vertex of the graph $\G$,
but is \emph{not} a point of the metric space $\G$ (this is consistent with
\eqref{H0}).

\begin{figure}[t]
\begin{center}
\begin{tikzpicture}
[scale=1.3,style={circle,
              inner sep=0pt,minimum size=7mm}]
\node at (0,0) [nodo] (1) {};
\node at (-1.5,0) [infinito]  (2){$\infty$};
\node at (1,0) [nodo] (3) {};
\node at (0,1) [nodo] (4) {};
\node at (-1.5,1) [infinito] (5) {$\infty$};
\node at (2,0) [nodo] (6) {};
\node at (3,0) [nodo] (7) {};
\node at (2,1) [nodo] (8) {};
\node at (3,1) [nodo] (9) {};
\node at (4.5,0) [infinito] (10) {$\infty$};
\node at (5.5,0) [infinito] (11) {$\infty$};
\node at (4.5,1) [infinito] (12) {$\infty$};
\draw [-] (1) -- (2) ;
 \draw [-] (1) -- (3);
 \draw [-] (1) -- (4);
 \draw [-] (3) -- (4);
 \draw [-] (5) -- (4);
 \draw [-] (3) -- (6);
 \draw [-] (6) -- (7);
 \draw [-] (6) to [out=-40,in=-140] (7);
\draw [-] (3) to [out=10,in=-35] (1.4,0.7); 
\draw [-] (1.4,0.7) to [out=145,in=100] (3); 
 \draw [-] (6) to [out=40,in=140] (7);
 \draw [-] (6) -- (8);
 \draw [-] (6) to [out=130,in=-130] (8);
 \draw [-] (7) -- (8);
 \draw [-] (8) -- (9);
  \draw [-] (7) -- (9);
  \draw [-] (9) -- (12);
  \draw [-] (7) -- (10);
  \draw [-] (7) to [out=40,in=140] (11);
\end{tikzpicture}
\end{center}
\caption{A metric graph with $5$ half-lines and $13$ bounded edges, one of which
forms a self-loop.}\label{Fig_gen}
\end{figure}

With $\G$ as above, a function $u:\G\to\R$ can be regarded as a
bunch of functions $(u_e)_{e\in E}$, where $u_e:I_e\to\R$ is the
restriction of $u$ to the edge $I_e$.
Endowing each edge $I_e$ with Lebesgue measure, one can define
$L^p$ spaces over $\G$ in the natural way, with norm
\[
\| u\|_{L^p(\G)}^p=\sum_{e\in E} \| u_e\|_{L^p(I_e)}^p,\quad
u=(u_e).
\]
Similarly, the Sobolev space $H^1(\G)$ is defined
as the set
of those functions $u:\G\to\R$ such that
\begin{equation}
\label{eq:18}
\text{$u=(u_e)$ is continuous on $\G$, and
$u_e\in H^1(I_e)$
for every edge $e\in E$,}
\end{equation}
with the natural norm
\[
\Vert u\Vert_{H^1(\G)}^2=\int_{\G}
\bigl(|u'(x)|^2+|u(x)|^2\bigr)\,dx=
\sum_{e\in E} \int_{I_e} \bigl(|u_e'(x_e)|^2
+|u_e(x_e)|^2\bigr)\,dx_e.
\]
Note that $H^1(\G)$ can be identified
with a closed subspace (determined by the continuity of $u$
at the vertices of $\G$)
of the Cartesian product $\bigoplus_{e} H^1(I_e)$.
In terms of the coordinate
system $\{x_e\}$, continuity on $\G$ means that, whenever
two edges $e,f$ meet at a vertex $\vv$ of $\G$, the corresponding branches
of $u$ satisfy a no-jump condition of the kind $u_e(0)=u_f(0)$ (or
$u_e(\ell_e)=u_f(\ell_f)$, or $u_e(\ell_e)=u_f(0)$ etc., depending on the orientation
of $I_e,I_f$ induced
by the coordinates $x_e,x_f$).
Notice that, according to \eqref{H0}, vertices at infinity are never involved in these continuity
conditions: on the other hand, if $u\in H^1(\G)$,
then automatically
\begin{equation}
\label{zeroallinfinito}
I_e=[0,+\infty)\implica
\lim_{x_e\to +\infty} u_e(x_e)=0,
\end{equation}
because in particular $u_e\in H^1(I_e)$.

\medskip

\noindent Within this framework, we are now in a position to state our main
results. Fix $\G$ as above, and numbers $\mu,p$ satisfying
\begin{equation}
\label{mup}
\mu>0\quad\text{and}\quad 2<p<6.
\end{equation}
For $u\in H^1(\G)$, the NLS energy in \eqref{NLSe} is finite and takes the concrete form
\[
E(u,\G)=\frac 1 2 \sum_e \int_{I_e} |u_e'(x_e)|^2\,dx_e-
\frac 1 p\sum_e \int_{I_e} |u_e(x_e)|^p\,dx_e.
\]
If we let
\begin{equation}
\label{h1m}
H_\mu^1 (\G)  : =  \bigl\{ u \in H^1 (\G)\quad\text{such that}\quad \| u \|_{L^2(\G)}^2 =
\mu \bigr\},
\end{equation}
the minimization problem \eqref{minE} takes the compact form
\begin{equation}
\label{minE2}
\min_{u\in H^1_\mu(\G)} E(u,\G).
\end{equation}

\begin{remark}\label{remclassic}
%
The classical instance of \eqref{minE2} where $\G$ is the real line $\R$
falls within our framework
as a particular case, when
$\G$ is
made up of \emph{two} unbounded edges (half-lines), joined at their initial point
(Figure~\ref{figtorre}.a).

In this case, the solutions to \eqref{minE2}
are called \emph{solitons}, and are known to be unique up to translations
and a change of sign. In particular,
there is a unique minimizer which is a \emph{positive and even function}: we shall denote
this function by $\phi_\mu$ (the dependence on $p$ being understood), and of course
\begin{equation}
\label{ensol}
E(\phi_\mu,\R)=\min_{\phi\in H^1_\mu(\R)} E(\phi,\R)<0.
\end{equation}
It is well known  that solitons obey the scaling rule
\begin{equation}
\label{scalingrule}
\phi_\mu(x)=
\mu^\alpha \phi_1\bigl ( \mu^\beta x\bigr ),\quad
\alpha=\frac 2{6-p},\quad
\beta=\frac {p-2} {6-p}
\end{equation}
where $\alpha,\beta>0$ by \eqref{mup}, and
$\phi_1(x)=C_p \mathop{\rm sech}(c_p x)^{\alpha/\beta}$ with $C_p,c_p>0$.
\end{remark}
The following is a general result for non-compact graphs.
\begin{theorem}\label{teoremzero}
If $\G$ contains at least one half-line, then
\begin{equation}
\label{eq:14bi}
\inf_{u\in H^1_\mu(\G)} E(u,\G)\leq
\min_{\phi\in H^1_\mu(\R)} E(\phi,\R)=
E(\phi_\mu,\R)
\end{equation}
and, in the other direction,
\begin{equation}
\label{eq:14tri}
\inf_{u\in H^1_\mu(\G)} E(u,\G)\geq
\min_{\phi\in H^1_\mu(\R^+)} E(\phi,\R^+)=
\frac 1 2 E(\phi_{2\mu},\R).
\end{equation}
\end{theorem}

In order to investigate
whether the infimum in \eqref{eq:14bi} is attained or not, the following
structure assumption on the graph $\G$ will play a crucial role:

\begin{itemize}
\item[(H)] After removal of any edge
$e\in E$, \emph{every} connected component of the \emph{graph} $(V,E\setminus\{e\})$ contains
at least one vertex $\vv\in V_\infty$.
\end{itemize}

Some remarks are in order. Firstly, (H) entails that $\G$ has \emph{at least} one
vertex at infinity $\vv_1\in V_\infty$, whence $\G$ is not compact
by \eqref{comp}. Secondly, the condition on
$e$ is relevant only when $e$ is a \emph{cut-edge} for $\G$ (i.e. when the
removal of $e$ disconnects $\G$), because when $(V,E\setminus\{e\})$
is connected the presence of $\vv_1\in V_\infty$ makes the condition trivial. On the other hand,
the edge $e$ (half-line) that has $\vv_1$ as vertex at infinity is necessarily
a cut-edge, since by \eqref{H0} its removal leaves vertex $\vv_1$ \emph{isolated}
in the graph $(V,E\setminus\{e\})$: therefore, the other connected component
necessarily contains a vertex at infinity $\vv_2\not=\vv_1$. Hence we see that,
in particular,
\begin{equation}
\label{almenodue}
\text{(H)}\quad\Rightarrow\quad
\text{$\G$ has at least two vertices at infinity.}
\end{equation}
Roughly speaking, assumption (H) says that there is always a vertex at infinity
on \emph{both sides} of any cut-edge.
Any cut-edge $e$ that violates (H),
would therefore
leave \emph{all} the vertices at
infinity on the same connected component
thus forming a sort of ``bottleneck'',
as regards the location of $V_\infty$ relative to $e$.
Thus,
in a sense, we may consider (H) as a no-bottleneck condition on $\G$.
Finally,
the fact that
(H) concerns cut-edges only makes it easy to test algorithmically, when the topology
of $\G$ is intricate: for instance, one can easily check that the graph in Figure~\ref{Fig_gen}
satisfies (H).

Under assumption (H), the inequality in \eqref{eq:14bi} is in fact an equality:
\begin{theorem} \label{teoremunoA}
If $\G$ satisfies (H), then
\begin{equation}
\label{infuguale}
\inf_{u\in H^1_\mu(\G)} E(u,\G)=
\min_{\phi\in H^1_\mu(\R)} E(\phi,\R)=
E (\phi_\mu, \R).
\end{equation}
\end{theorem}
\begin{remark*}
Hypothesis (H) and hence Theorem \ref{teoremunoA} apply to
examples of graphs
previously treated in the literature: star-graphs with
unbounded edges (\cite{acfn12a}) and general multiple bridges
(\cite{acfn14}). Furthermore,
they apply to any
semi-eulerian graph
with two vertices at infinity,  as well as
to more complicated networks like the
one represented in Figure \ref{Fig_gen}.
\end{remark*}

It is easy to construct examples of graphs $\G$ satisfying (H),
for which the infimum in \eqref{infuguale} is achieved.

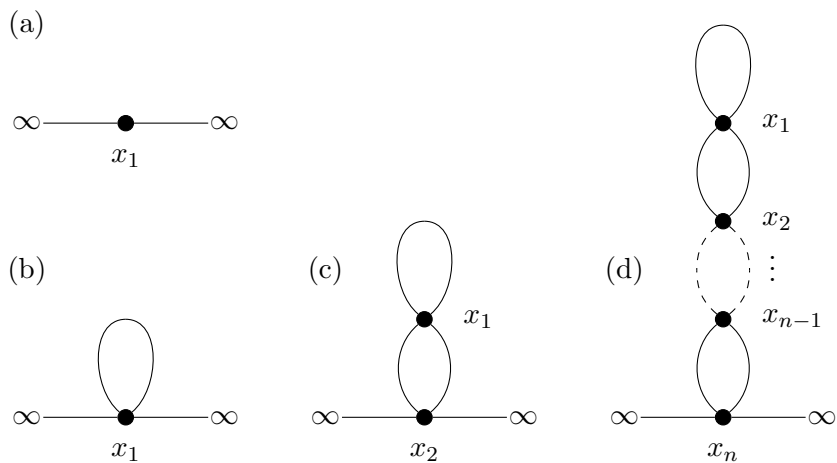
\begin{figure}[h]
\begin{center}
\begin{tikzpicture}
[scale=1.3,style={circle,
              inner sep=0pt,minimum size=7mm}]
\node at (0,3) [nodo] [label=below:$x_1$]  (OO)  {};
\node at (-1,3) [infinito] (RR) {\iii};
\node at (1,3) [infinito] (LL) {\iii};
\draw [-] (LL) -- (OO) ;
\draw [-] (RR) -- (OO) ;
\draw (-1,4) node {(a)};
\node at (0,0) [nodo] [label=below:$x_1$]  (O)  {};
\node at (-1,0) [infinito] (R) {\iii};
\node at (1,0) [infinito] (L) {\iii};
\draw [-] (L) -- (O) ;
\draw [-] (R) -- (O) ;
 \draw [-] (O) to [out=40,in=0] (0,1);
 \draw [-] (O) to [out=140,in=180] (0,1);
\draw (-1,1.5) node {(b)};
\end{tikzpicture}\quad\quad
\begin{tikzpicture}
[scale=1.3,style={circle,
              inner sep=0pt,minimum size=7mm}]
\node at (0,0) [nodo] [label=below:$x_2$] (O)  {};
\node at (1,0) [infinito] (L) {\iii};
\node at (-1,0) [infinito] (R) {\iii};
\node at (0,1) [nodo] [label=right:$\quad x_1$] (1)  {};
\draw [-] (L) -- (O) ;
\draw [-] (R) -- (O) ;
\draw [-] (1) to [out=40,in=0] (0,2);
\draw [-] (1) to [out=140,in=180] (0,2);
\draw [-] (O) to [out=40,in=-40] (1);
\draw [-] (O) to [out=140,in=-140] (1);
\draw (-1,1.5) node {(c)};
\end{tikzpicture}\quad\quad
\begin{tikzpicture}
[scale=1.3,style={circle,
              inner sep=0pt,minimum size=7mm}]
\node at (0,0) [nodo] [label=below:$x_n$] (O)  {};
\node at (1,0) [infinito] (L) {\iii};
\node at (-1,0) [infinito] (R) {\iii};
\node at (0,1) [nodo] [label=right:$\quad x_{n-1}$] (1)  {};
\node at (0,3) [nodo] [label=right:$\quad x_{1}$] (3)  {};
\node at (0,2) [nodo] [label=right:$\quad x_{2}$] (2)  {};
\draw (0.5,1.6) node {$\vdots$};
\draw [-] (L) -- (O) ;
\draw [-] (R) -- (O) ;
\draw [-] (3) to [out=40,in=0] (0,4);
\draw [-] (3) to [out=140,in=180] (0,4);
\draw [-] (O) to [out=40,in=-40] (1);
\draw [-] (O) to [out=140,in=-140] (1);
\draw [-] (2) to [out=40,in=-40] (3);
\draw [-] (2) to [out=140,in=-140] (3);
\draw [dashed] (1) to [out=40,in=-90] (0.267,1.5);
\draw [dashed] (1) to [out=140,in=-90] (-0.267,1.5);
\draw [dashed] (2) to [out=-40,in=90] (0.267,1.5);
\draw [dashed] (2) to [out=-140,in=90] (-0.267,1.5);
  %
\draw (-1,1.5) node {(d)};
\end{tikzpicture}
\end{center}
\caption{Graphs described in Example \ref{example1}, for which \eqref{eq:14bi} is an equality.}
\label{figtorre}
\end{figure}

\begin{example}\label{example1} \emph{(a)}
If $\G$ is isometric to $\R$ (see Remark~\ref{remclassic} and Figure~\ref{figtorre}.a), then
the soliton $\phi_\mu$ can be seen as an element of $H^1_\mu(\G)$,
and by \eqref{ensol} the infimum in \eqref{infuguale} is achieved.

\medskip

\noindent\emph{(b)} The symmetry of the soliton $\phi_\mu\in H^1(\R)$ can be exploited to
construct other examples. Given $a_1>0$, let $\G$ be
the quotient space $\R/\{\pm a_1\}$, obtained by gluing together the two points $a_1$ and
$-a_1$ into a unique point $x_1$. As a metric graph, $\G$ is depicted in
Figure~\ref{figtorre}.b, the length of the loop being $2a_1$. Since $\phi_\mu(a_1)=\phi_\mu(-a_1)$,
$\phi_\mu$ can be seen as an element of $H^1_\mu(\G)$, letting $x=0$ correspond to the north pole
of the loop in Figure~\ref{figtorre}.b. As before,
by \eqref{ensol} the infimum in \eqref{infuguale} is achieved.

\medskip

\noindent\emph{(c)} More generally, for $n\geq 2$
fix $a_n>\ldots> a_1>0$,  and let $\G$ be obtained from $\R$ by gluing together
each pair of points $\pm a_1,\ldots,\pm a_n$, the corresponding new points
being denoted $\{x_j\}$. As a metric graph, $\G$ is as
in Figure~\ref{figtorre}.d (the length of the loop at the top being $2a_1$,
while the pairs of parallel edges have lengths $a_j-a_{j-1}$, $2\leq j\leq n$).
Since $\phi_\mu(a_i)=\phi_\mu(-a_i)$,
reasoning as in (b)
we see that the infimum in \eqref{infuguale} is attained.$\quad\square$
\end{example}
In fact, the graphs of the previous example are the \emph{only ones} for which
the infimum is attained.
\begin{theorem} \label{teoremunoB}
If $\G$ satisfies (H) then \eqref{infuguale} holds
true,
but the infimum is never achieved unless $\G$
is isometric to one of the graphs discussed in Example~\ref{example1}.
\end{theorem}
Thus, with the only exception of the graphs of Example~\ref{example1},
assumption (H) rules out the existence of minimizers. Among metric graphs with
at least two half-lines, the simplest one that violates (H) is the graph in
Figure~\ref{Figbridge}.b, made up of two half-lines and one bounded edge (of arbitrary length)
joined at their initial point.

For this graph, we have the following result.
\begin{theorem} \label{teoremdue}
Let $\G$ consist of
two half-lines and one bounded edge (of arbitrary length $\ell>0$)
joined at their initial point. Then
\begin{equation}
\label{infminore}
\inf_{u\in H^1_\mu(\G)} E(u,\G)<
\min_{\phi\in H^1_\mu(\R)} E(\phi,\R)=E(\phi_\mu,\R)
\end{equation}
and the infimum is achieved.
\end{theorem}
As mentioned in the introduction, any minimizer exploits
the peculiar topology of this graph, and tends to concentrate on
the pendant. This is described in the following theorem.
\begin{theorem} \label{teoremdueB}
Let $\G$ be as in Theorem~\ref{teoremdue}, and let
$u\in H^1_\mu(\G)$ be any minimizer that achieves the infimum in \eqref{infminore}.
Then, up to replacing $u$ with $-u$, we have $u>0$ and
\begin{itemize}
\item[(i)] $u$ is strictly monotone along the pendant, with a
maximum at the tip.
\item[(ii)] If $u_1,u_2$ denote the restrictions of $u$ to the two half-lines,
with coordinates $x\geq 0$ starting both ways at the triple junction, then
\[
u_1(x)=u_2(x)=\phi_{\mu^*}(x+y)\quad\forall x\geq 0,
\]
for suitable $y>0$ and $\mu^*>\mu$ that depend on the mass $\mu$ and the length of
the pendant $\ell$. In particular, the restriction of $u$
to the straight line
is symmetric and radially decreasing, with a corner point at the origin.
\item[(iii)] For fixed $\mu$, the infimum in \eqref{infminore} is a strictly decreasing function
of $\ell$.
\end{itemize}
\end{theorem}
From (i) and (ii) it follows that the minimum of $u$ along the pendant coincides with
its maximum on the straight line (in other words, $u$ tends to concentrate on the pendant). Observe that,
on each half-line, $u$ coincides with a suitable
\emph{portion} of the soliton $\phi_{\mu^*}$.

\section{Some preliminary results}\label{sec:prelim}

The \emph{decreasing rearrangement} $u^*$ of a function $u\in H^1(\G)$,
where $\G$ is a metric graph, was first used in \cite{Friedlander} where
it is proved that, as in the classical case where $\G$ is an interval (see \cite{Kawohl}),
this kind of rearrangement does not increase the Dirichlet integral (see also \cite{acfn12b}).
Besides the increasing
rearrangement $u^*$, we shall also need the \emph{symmetric rearrangement}
$\widehat{u}$, whose basic properties we now recall.

Given $u\in H^1(\G)$, assume for simplicity that
\begin{equation}
\label{eq:90}
m:=\inf_{\G} u\geq 0,\qquad M:=\sup_{\G} u>0
\end{equation}
and, as in \cite{Friedlander}, let $\rho(t)$ denote the distribution
function of $u$:
\[
\rho(t)=\sum_{e\in E}\mis\bigl(\{x_e\in I_e\,:\,\,
u_e(x_e)>t\}\bigr),\quad t\geq 0,
\]
where the $u_e$'s are
the branches of $u$ as  in \eqref{eq:18}. Set
\begin{equation}
\label{defomega}
\omega:=\sum_{e\in E} \mis(I_e),\quad
I^*:=[0,\omega),\quad
\widehat{I}:=(-\omega/2,\omega/2)
\end{equation}
where $\omega\in [0,\infty]$ is the total length of $\G$. As usual,
one can define
the following rearrangements of $u$:
\begin{enumerate}
\item [(i)]the \emph{decreasing} rearrangement $u^*:I^*\to \R$ as the function
\begin{equation}
\label{defu*dec}
u^*(x):=\inf\{t\geq 0\,:\,\, \rho(t)\leq x\},\quad x\in I^*;
\end{equation}
\item [(ii)] the \emph{symmetric decreasing} rearrangement $\widehat{u}:\widehat{I}\to \R$ as the function
\begin{equation*}
\widehat{u}(x):=\inf\{t\geq 0\,:\,\, \rho(t)\leq 2|x|\},\quad x\in \widehat{I}.
\end{equation*}
\end{enumerate}
Since $u$, $u^*$ and $\widehat u$ are equimeasurable, one has
\begin{equation}
\label{normelr}
\int_{I^*} |u^*(x)|^r\,dx=
\int_{\widehat{I}} |\widehat{u}(x)|^r\,dx=
\int_{\G} |u(x)|^r\,dx
\quad\forall r>0
\end{equation}
and
\begin{equation*}
\inf_{I^*}u^*=
\inf_{\widehat{I}}\widehat{u}=\inf_{\G}u=m,\quad
\sup_{I^*}u^*=
\sup_{\widehat{I}}\widehat{u}=\sup_{\G}u=M.
\end{equation*}
As in the classical case
where $\G$ is an interval (see \cite{Kawohl}),
when $\G$ is a
\emph{connected} metric graph
it turns out (see \cite{Friedlander}) that
 $u^*\in H^1(I^*)$ and
$\widehat{u}\in H^1(\widehat{I})$ respectively (connectedness of $\G$ is not essential, as long
as the image of $u$ is connected).
However,
while the passage from $u$ to $u^*$ never increases the
Dirichlet integral (\cite{Friedlander}), this is not always true
for $\widehat{u}$, a sufficient condition being that  the number of preimages
\begin{equation*}
N(t):=\#\{x\in \G\,:\,\, u(x)=t\},
\quad t\in (m,M)
\end{equation*}
is \emph{at least two} (see
Remark~2.7 in \cite{Kawohl}).
More precisely, we have
\begin{proposition}
\label{proprear}
Let $\G$ be a connected metric graph, and let $u\in H^1(\G)$ satisfy \eqref{eq:90}.
Then
\begin{equation}
\label{eq:ens}
\int_{I^*} |(u^*)'|^2\,dx\leq\int_{\G} |u'|^2\,dx,
\end{equation}
with strict inequality unless $N(t)=1$ for a.e. $t\in (m,M)$. Finally,
\begin{equation}
\label{eq:ens2}
N(t)\geq 2\quad\text{for a.e. $t\in (m,M)$}
\quad\Rightarrow\quad\int_{\widehat{I}} |(\widehat{u})'|^2\,dx\leq\int_{\G} |u'|^2\,dx,
\end{equation}
where equality implies that $N(t)=2$ for a.e. $t\in (m,M)$.
\end{proposition}
The part concerning $u^*$ can be found in \cite{Friedlander},
while the corresponding statements for $\widehat u$ can be proved in exactly
the same way.

\begin{remark}\label{remomega}
If $\G$ is non-compact, i.e. if $\G$ contains at least one half-line,
then clearly $\omega=+\infty$ in \eqref{defomega}, so that
$I^*=\R^+$ and $\widehat{I}=\R$. Thus, in particular, $u^*\in H^1(\R^+)$
while $\widehat{u}\in H^1(\R)$.
\end{remark}

The following standard result deals with the optimality conditions satisfied
by
any solution to \eqref{minE2}.
\begin{proposition}
\label{necess} Let $\G$ be a metric graph, and $u \in H_\mu^1(\G)$
a solution to \eqref{minE2}.
Then
\begin{itemize}
\item[(i)] there exists $\lambda \in\R$ such that
\begin{equation}
\label{eulero}
u_e'' +u_e |u_e|^{p-2} = \lambda u_e\quad\text{for every edge $e$;}
\end{equation}
\item[(ii)] for every vertex $\vv$ (that is not a vertex at infinity)
\begin{equation}
\label{kir}
\sum_{e\succ \vv}
\frac{d u_e}{d x_e}(\vv)=0
\qquad\text{ (Kirchhoff conditions),}
\end{equation}
where the condition $e\succ \vv$ means that edge $e$ is incident at $\vv$;
\item[(iii)] up to replacing $u$ with $-u$, one has that $u>0$ on $\G$.
\end{itemize}
\end{proposition}
The Kirchhoff condition \eqref{kir} is well known (see \cite{Friedlander,Kuchment}) and
is a natural form of continuity of $u'$ at the vertices of $\G$.
Observe that, by \eqref{eulero}, $u_e\in H^2(I_e)$ for every edge $e$, so that $u'_e$ is well defined
at both endpoints of $I_e$: in \eqref{kir}, of course, the symbol
$d u_e/d x_e(\vv)$ is a shorthand notation for $u_e'(0)$ or $-u_e'(\ell_e)$, according to
whether the coordinate $x_e$ is equal to $0$ or $\ell_e$ at $\vv$.

\begin{proof} Since both the energy $E(u,\G)$ and the $L^2$ constraint in
\eqref{h1m} are differentiable in $H^1(\G)$
and $u$ is a constrained critical point, computing G\^{a}teaux derivatives
one has
\begin{equation}
\label{weak}
\int_\G \left(u'\eta' - u |u|^{p-2} \eta\right)\dx
+ \lambda \int_\G u \eta\dx =0\qquad \forall \eta \in H^1(\G)
\end{equation}
where $\lambda$ is a Lagrange multiplier.
Fixing an edge $e$, choosing $\eta\in C^\infty_0(I_e)$ and integrating by parts,
one obtains \eqref{eulero}.

Now fix a vertex $\vv$ (not at infinity) and choose $\eta\in H^1(\G)$,
null at every vertex of $\G$ except at $\vv$: integrating by parts in \eqref{weak}
and using (i), only the boundary terms at $\vv$ are left, and one finds
\[
-\sum_{e\succ \vv}
\frac{d u_e}{d x_e}(\vv)\eta(\vv)=0,
\]
and \eqref{kir} follows since $\eta(\vv)$ is arbitrary.

To prove (iii), observe that if $u$ is a minimizer so is $|u|$, hence we may assume
that $u\geq 0$.
First assume that $u$ vanishes at a vertex
$\vv$.
Since $u\ge 0$ on $\G$,
no term
involved in \eqref{kir}
can be negative: since their sum is zero,
every derivative in \eqref{kir} is in fact zero.
Then, by uniqueness for the ODE \eqref{eulero}, we see that
$u_e\equiv 0$ along every edge $e$ such that $e\succ \vv$: since
$\G$ is connected, this argument can be iterated through neighboring vertices
and one obtains that $u\equiv 0$ on $\G$, a contradiction since $u\in H^1_\mu(\G)$.
If, on the other hand, $u_e(x)=0$ at some point $x$ interior to some edge $e$, from $u\geq 0$
we see that also $u'_e(x)=0$ and, as before, from \eqref{eulero} we deduce that
$u_e\equiv 0$ along $e$. Thus, in particular, $u(\vv)=0$ at a vertex $\vv \prec e$,
and one can argue as above.
\end{proof}

Another useful result, valid for any metric graph $\G$, is the
following Gagliardo-Nirenberg inequality:
\begin{equation*}
\|u\|_{L^p(\G)}^p \le C\|u\|_{L^2(\G)}^{\frac p 2 +1}
\|u\|_{H^1(\G)}^{\frac p 2-1}\qquad \forall u \in H^1(\G),
\end{equation*}
where $C=C(\G,p)$. This is well known when $\G$ is an
interval (bounded or not, see \cite{brezis}): for the general case,
it suffices to write the inequality for each edge of $\G$, and take the sum.

In particular, when $\|u\|_{L^2(\G)}^2=\mu$ is fixed, we obtain
\begin{equation*}
\|u\|_{L^p(\G)}^p \le C+
C\|u'\|_{L^2(\G)}^{\frac p 2-1}\qquad \forall u \in H^1_\mu(\G),
\end{equation*}
where now $C=C(\G,p,\mu)$. Since our $p$ satisfies \eqref{mup}, this shows
that the negative term in \eqref{NLSe} grows \emph{sublinearly}, at infinity,
with respect to the positive one. As  a consequence, Young's inequality gives
$\|u'\|_{L^2(\G)}^2\leq C+CE(u,\G)$ when $u \in H^1_\mu(\G)$, and hence also
\begin{equation}
\label{coerc}
\|u\|_{H^1(\G)}^2 \leq C+CE(u,\G)
\qquad \forall u \in H^1_\mu(\G),\quad
C=C(\G,p,\mu).
\end{equation}

\section{Some auxiliary results}\label{sec:aux}
In this section we discuss two auxiliary
\emph{double-constrained} problems,
on $\R^+$ and on $\R$ respectively, that
will be useful in Section~\ref{sec:baffo}
and may be of some interest in themselves.

We begin with the double-constrained problem on the half-line
\begin{equation}
\label{minhalf}
\min E(\phi,\R^+),\quad
\phi \in H^1(\R^+),
\quad
\int_0^\infty |\phi|^2\,dx=\frac m 2,\quad
\phi(0)=a
\end{equation}
for fixed $m,a>0$. This corresponds to \eqref{minE} when $\G=\R^+$ and $\mu=m/2$,
with the
\emph{additional} Dirichlet condition
\begin{equation*}
\phi(0) = a.
\end{equation*}

\begin{theorem}
\label{unic}
For every $a,m>0$  there exist unique
$M>0$ and $y\in\R$ such that the soliton $\phi_M$ satisfies the two conditions
\begin{equation}
\label{cond}
\phi_{M}(y) =a\qquad\text{and}\qquad \int_0^{+\infty} \phi_{M}(y+x)^2\dx = \frac m 2.
\end{equation}
Moreover, the function $x\mapsto\phi_M(y+x)$ is the unique solution to \eqref{minhalf}.

Finally, there holds
\begin{equation}
\label{ypos}
a>\phi_m(0)\iff y>0.
\end{equation}
\end{theorem}

\begin{proof}
Recalling  the scaling rule of solitons
\eqref{scalingrule},
let $z=\mu^\beta y$ (to be determined).
Then, changing variable $x=M^{-\beta} t$ in the integral, the conditions in \eqref{cond} become
\begin{equation}
\label{condbis}
M^\alpha \phi_1(z)=a,\quad\text{and}\quad
M \int_0^{+\infty} \phi_1(z+t)^2\dt = \frac m 2.
\end{equation}
Using the first condition, we can eliminate $M$ from the second and obtain
\begin{equation}
\label{tosolve}
\phi_1(z)^{-\frac{1}\alpha} \int_0^{+\infty} \phi_1(z+t)^2\dt = \frac{m a^{-\frac{1}\alpha}}{2}.
\end{equation}
Denoting by $g(z)$ the function on the left-hand side, we claim that
\begin{equation}
\label{limiti}
\lim_{z\to -\infty}g(z) = +\infty \qquad\hbox{and}\qquad \lim_{z\to +\infty}g(z) =0.
\end{equation}
The first limit is clear as
$\phi_1(z)\to 0$ while the integral tends to $\| \phi_1\|_{L^2(\R)}^2=1$.
For the second, since $\phi_1$ is decreasing on $\R^+$ one can estimate
\[
\phi_1(z+t)^2\leq
\phi_1(z)^{\frac 1 \alpha}
\phi_1(z+t)^{2-\frac 1 \alpha}\quad\forall z,t\geq 0
\]
in the integral and observe that $2-1/\alpha>0$ by \eqref{mup}.

Moreover, $g$ is strictly decreasing: this is clear when $z<0$,
since in this case $g(z)$ is the product of two strictly decreasing functions.
When $z>0$,
differentiation yields
\begin{align*}
g'(z) &= 2\int_0^{+\infty}
\frac{\phi_1(z+t)^2}{\phi_1(z)^{\frac 1\alpha}}
\left[\frac{\phi_1'(z+t)}{\phi_1(z+t)}-
\frac 1{2\alpha}\frac{\phi_1'(z)}{\phi_1(z)}\right]\dt\\
&<
2\int_0^{+\infty}
\frac{\phi_1(z+t)^2}{\phi_1(z)^{\frac 1\alpha}}
\left[\frac{\phi_1'(z+t)}{\phi_1(z+t)}-
\frac{\phi_1'(z)}{\phi_1(z)}\right]\dt
<0,
\end{align*}
having used $1/2\alpha<1$ and $\phi_1'(z)<0$ in the first inequality, and
the log--concavity of $\phi_1$ (see Remark~\ref{remclassic})
in the second. This and \eqref{limiti} show that,
given $a,m>0$, there exists a unique $z\in \R$ (hence a unique $y$)
satisfying
\eqref{tosolve}, while $M$ is uniquely determined by the first
condition in \eqref{condbis}.

To prove \eqref{ypos} observe that, by \eqref{scalingrule}, $a>\phi_m(0)$
is equivalent to $a>m^\alpha\phi_1(0)$, which in turn is equivalent to
$g(0)>ma^{-1/\alpha}/2$ since
\[
g(0)=
\phi_1(0)^{-\frac{1}\alpha} \int_0^{+\infty} \phi_1(t)^2\dt =
\frac{\phi_1(0)^{-\frac{1}\alpha}}2.
\]
But since $g(z)=ma^{-1/\alpha}/2$ by \eqref{tosolve} and $g$ is decreasing, the last inequality
is equivalent to $z>0$, which proves \eqref{ypos}.

For the last part of the claim,
by Remark~\ref{remclassic}
we see that the function
$x\mapsto \phi_M(x+y)$ minimizes $E(\phi,\R)$
under the \emph{two constraints} $\Vert\phi\Vert_{L^2(\R)}^2=M$ and
 $\phi(0)=a$. If a competitor $\varphi(x)$
better than $x\mapsto \phi_M(x+y)$ could be found for \eqref{minhalf}, then
the function equal to $\phi_M(x+y)$ for $x<0$ and to $\varphi(x)$ for $x\geq 0$ would
violate the mentioned optimality of $\phi_M(x+y)$. Finally, uniqueness
follows from the uniqueness of $M$ and $y$ satisfying \eqref{cond}: indeed,
any other solution $\varphi(x)$  to \eqref{minhalf} not coinciding on $\R^+$ with
any soliton,  arguing as before would give rise to a \emph{non-soliton} minimizer
of $E(\phi,\R)$ with mass constraint $\Vert\phi\Vert_{L^2(\R)}^2=M$.

\end{proof}

\noindent Now  we consider the analogue problem on the \emph{whole} real line, namely
\begin{equation}
\label{pretta}
\min E(\phi,\R),\quad
\phi \in H^1(\R),
\quad
\int_{-\infty}^\infty |\phi|^2\,dx=m,\quad
\phi(0)=a
\end{equation}
for fixed $m,a>0$. Its solutions are characterized as follows, a
special role being played by the soliton $\phi_m$ of mass $m$.

\begin{theorem}
\label{thretta} Let $a,m > 0$ be given.
\begin{enumerate}
\item[(i)] If $a<\phi_m(0)$ then problem \eqref{pretta}
  has exactly two solutions, given by $x\mapsto \phi_m(x\pm y)$ for a suitable $y>0$.
\item[(ii)]  If $a=\phi_m(0)$, then problem
  \eqref{pretta} has $\phi_m(x)$ as unique solution.
\item[(iii)]  if $a>\phi_m(0)$, then
problem \eqref{pretta} has exactly one solution, namely
$x\mapsto \phi_M(|x|+y)$ for suitable $M, y>0$.
\end{enumerate}
\end{theorem}

\begin{proof} Cases (i) and (ii) are immediate since  $a$ is in the range of $\phi_m$:
as $\phi_m$ and its translates are the only positive minimizers of $E(\cdot,\R)$ in $H^1_m(\R)$,
the second constraint in \eqref{pretta}
can be matched for free by
a translation, with $y\geq 0$ such that $\phi_m(\pm y)=a$.
Case (ii) is when $a=\max\phi_m$, and so $y=0$.

Now consider (iii), where the value $a$ is not in the range of $\phi_m$.
Let $M,y$ be the numbers provided by Theorem~\ref{unic}, and observe that
$y>0$ according to \eqref{ypos}. Since $\phi_M(y)=a$ by
\eqref{cond}, the soliton $\phi_M$ is clearly the
unique solution of the constrained problem
\begin{equation}
\label{aux1}
\min E(\phi,\R),\quad
\phi\in H^1_M(\R), \quad
\phi(-y)=\phi(y)=a.
\end{equation}
Moreover, as $y>0$, the second condition in \eqref{cond} implies that
\[
\int_{\R\setminus (-y,y)}|\phi_M(x)|^2\,dx
=
2\int_{y}^\infty |\phi_M(x)|^2\,dx=m.
\]
Therefore, the function $w(x)=\phi_M(|x|+y)$ is an admissible competitor
for \eqref{pretta}, and obviously
\[
E(w,\R)=E\bigl(\phi_M,(-\infty,-y)\bigr)
+E\bigl(\phi_M,(y,+\infty)\bigr).
\]
The existence of a $v(x)$ admissible for \eqref{pretta} and such that
$E(v,\R)<E(w,\R)$, would allow the construction of a competitor better than $\phi_M$
in \eqref{aux1}, by redefining $\phi_M(x)$ when $|x|\geq y$, setting it
equal to $v(x-y)$ or $v(x+y)$, according to whether $x\geq y$ or $x\leq -y$.
Thus $w$ solves \eqref{pretta}. Similarly, a competitor $v\not\equiv w$
with $E(v,\R)=E(w,\R)$ would violate the uniqueness of $\phi_M$ as a solution
of \eqref{aux1}, and therefore $w$ is the unique solution of \eqref{pretta}.
\end{proof}

\section{Proof of the nonexistence results}\label{sec:proofs1}
\begin{proof}[Proof of Theorem~\ref{teoremzero}]
Consider $u_\eps\in H^1(\R)$ with compact support,
such that $\Vert u_\eps\Vert_{L^2}^2=\mu$ and
$u_\eps\to \phi_\mu$ strongly in $H^1(\R)$ as $\eps\to 0$. Since
$u_\eps\to \phi_\mu$ also in $L^p(\R)$, we see that
\[
E(u_\eps,\R)\to E(\phi_\mu,\R)\quad
\text{as $\eps\to 0$.}
\]
Now, by a translation, we may
assume that $u_\eps$ is supported in $[0,+\infty)$: identifying this
interval with one of the half-lines
of $\G$, we may consider $u_\eps$ as a function in $H^1_\mu(\G)$, by extending it
to zero on any other edge of $\G$. Then we have from the previous equation
\[
\inf_{u\in H^1_\mu(\G)} E(u,\G) \leq
\lim_{\eps\to 0} E(u_\eps,\G)=
E(\phi_\mu,\R),
\]
and \eqref{eq:14bi} follows from \eqref{ensol}. The inequality in
\eqref{eq:14tri} is immediate from \eqref{normelr} and
\eqref{eq:ens} (see Remark~\ref{remomega}),
by rearranging an arbitrary $u\in H^1_\mu(\G)$. Finally, the
equality in
\eqref{eq:14tri} is well known.
\end{proof}
To  prove of Theorems \ref{teoremunoA} and \ref{teoremunoB}
we need to investigate
how assumption (H),
which is purely graph-theoretical,
reflects on $\G$ as a \emph{metric-space}.
\begin{lemma}\label{lemmaHprime}
Assume $\G$ is connected and satisfies condition (H). Then $\G$, as a metric space,
satisfies the following condition as well:
\begin{itemize}
\item[{\rm(H$'$)}] For every point $x_0\in\G$, there exist two injective curves
$\gamma_1,\gamma_2:[0,+\infty)\to\G$ parameterized by arclength, with disjoint
images except for finitely many points, and such that $\gamma_1(0)=\gamma_2(0)=x_0$.
\end{itemize}
\end{lemma}
\begin{proof}
We recall  that a \emph{trail} in a graph (see \cite{bollobas})
is a finite sequence of consecutive edges, in which no edge is repeated (note that
a vertex, instead, may be repeated, if the trail contains cycles  or self-loops).
Let ${\mathcal T}$ be the class of all those \emph{trails} $T$ in the graph $\G$,
such that the initial edge and the final edge of $T$
are half-lines (observe that, due to \eqref{H0}, no edge of $T$ other than the initial and the
final edge can be a half-line:
thus, in a sense, every $T\in{\mathcal T}$
provides an immersion of $\R$ in $\G$ ---not an embedding, however, as $T$ may have
cycles). Recalling \eqref{almenodue}, since $\G$ is connected we see that
${\mathcal T}\not=\emptyset$.
Moreover, any point $x_0\in\G$ that is covered by at least one trail $T\in{\mathcal T}$,
 satisfies condition (H$'$). Indeed, to construct $\gamma_1$
and $\gamma_2$, it suffices to start at $x_0$ and move along $T$ both ways,
according to arclength: the two obtained curves can then be made injective,
by removing any useless loop that each of them, separately, may form. Moreover, since no edge
is repeated in
$T$, $\gamma_1$ and $\gamma_2$ may intersect only at finitely many vertices of $\G$.
Since the initial and the final edges of $T$ are half-lines, each $\gamma_i(t)$
is eventually trapped in a half-line and thus each $\gamma_i$ has infinite length,
which makes
it possible to parameterize $\gamma_i(t)$ injectively by arclength with $t\in [0,+\infty)$.
Therefore, it suffices to prove that the trails in ${\mathcal T}$ cover
$\G$.

Let $E_0\subseteq E$ denote the set of those edges that do not belong to any
trail $T\in{\mathcal T}$. Assuming $E_0\not=\emptyset$,
since $\G$ is connected there must be some edge $e\in E_0$ with one vertex
$\ww$  on some trail $T\in{\mathcal T}$ ($\ww\in V\setminus V_\infty$,
by \eqref{H0}).
 If $e$ were a cut-edge for $\G$,
then by (H) there would be a vertex $\vv_1\in V_\infty$ in the connected component
of $(V,E\setminus\{e\})$ disjoint from $T$: then it would be possible to
go from $\vv_1$ to $\ww$  along a trail that crosses $e$, and then proceed over a portion of $T$
up to another vertex $\vv_2\in V_\infty$,  thus constructing a path
$T'\in{\mathcal T}$ that contains $e$, which is impossible since $e\in E_0$.
Now, as $e$ is not a cut-edge, it necessarily belongs to a cycle $C$ (made up of
bounded edges only, due to \eqref{H0}): then, by inserting $C$ in the middle of
$T$ at vertex $\ww$, we can construct a trail $T'\in {\mathcal T}$ that contains
$e$. Since this is a contradiction, we see that $E_0=\emptyset$ and therefore
${\mathcal T}$ covers $\G$.
\end{proof}

\begin{proof}[Proof of Theorem \ref{teoremunoA}]
We first prove that for every $u\in H^1_\mu(\G)$,
\begin{equation}
\label{eq:40} E(u,\G)\geq E(\widehat{u},\R)\geq \min_{\phi\in
H^1_\mu(\R)}E(\phi,\R)
\end{equation}
where $\widehat{u}$ is the symmetric rearrangement of $u$ as
defined in \eqref{defu*dec}.

As $E(u,\G)=E(|u|,\G)$, we may assume that $u\geq 0$. In fact we
have that $M:=\max_{\G} u>0$ by \eqref{mass}, and that
$m:=\inf_{\G} u=0$ by \eqref{zeroallinfinito}, as $\G$
contains at least two half-lines                  due to \eqref{almenodue}.
From Remark~\ref{remomega}  $\widehat{u}\in H^1(\R)$ and,
in fact, $\widehat{u}\in
H^1_\mu(\R)$ by \eqref{normelr} with $r=2$, which proves the second inequality
in \eqref{eq:40}. The first inequality,
due to \eqref{normelr} written with $r=p$,
is in fact equivalent
to
the integral inequality in \eqref{eq:ens2}, which
follows as soon as we show that
\begin{equation}
\label{eq:cpi}
N(t):=\#\{x\in \G\,:\,\, u(x)=t\} \geq
2\quad\text{for a.e. $t\in (0,M)$.}
\end{equation}
This, in turn, follows from (H$'$) of Lemma~\ref{lemmaHprime}.
If $\gamma_1,\gamma_2$ are as in (H$'$),
relative to
a point  $x_0\in\G$ where $u(x_0)=M$,
we may
define the continuous function
\begin{equation}
\label{defv}
v:\R\to\R,\quad
v(z)=\begin{cases}
u(\gamma_1(z)) & \text{if $z\geq 0$,}\\
u(\gamma_2(-z)) & \text{if $z< 0$.}\\
\end{cases}
\end{equation}
Clearly $v(0)=u(x_0)=M$. Moreover, as each $\gamma_i(z)$ parameterizes a half-line of
$\G$ for $z$ large enough,
from \eqref{zeroallinfinito} we have that
$v(z)\to 0$ as $|z|\to\infty$, hence
$v$
has at least two
distinct preimages (in $\R$) for every value $t\in (0,M)$.
But
as the images of $\gamma_1,\gamma_2$ are disjoint except for finitely many points of $\G$,
\eqref{eq:cpi} is established and \eqref{eq:40} follows. Now \eqref{eq:40}, combined
with \eqref{eq:14bi}, proves Theorem~\ref{teoremunoA}.
\end{proof}

\begin{proof}[Proof of Theorem \ref{teoremunoB}] Carrying on with the previous proof,
assume that some $u\in H^1_\mu(\G)$ achieves the infimum in \eqref{infuguale}
(i.e. in \eqref{eq:40}).
Then, both inequalities in \eqref{eq:40} are equalitites. From them
(combining the former with Proposition~\ref{proprear})
we infer that
\begin{itemize}
\item[(i)] $N(t)=2$ for a.e. $t\in (0,M)$,
\item[(ii)] $\widehat{u}$ is a soliton of mass $\mu$, i.e. $\widehat{u}=\phi_\mu$.
\end{itemize}
Now, if $\Gamma_i$ denotes the image of the curve $\gamma_i$
defined above,
we claim that
\begin{itemize}
\item[(iii)]
the union $\Gamma=\Gamma_1\cup\Gamma_2$ covers $\G$.
\end{itemize}
Indeed, in the proof of \eqref{eq:cpi}, we proved a slightly stronger
statement, namely that
$\#( u^{-1}(t)\cap\Gamma)\geq 2$
for a.e. $t$ (while $N(t)$ counts preimages in the whole $\G$).
Then, from (i),
it follows that
those values $t\in (0,M)$ attained by
$u$ on $\G\setminus\Gamma$ (if any) form a set of measure zero. Since $u$ is
continuous on $\G$ and $\G\setminus\Gamma$ is open ($\Gamma$ is, in fact,
the trail mentioned in assumption (H)),
this implies that $u$ is \emph{constant} on any edge $e$ of
$\G$ not belonging to $\Gamma$
(as $\Gamma$ is a trail, we may regard it is as a subgraph of $\G$).
But $u$ and $\widehat{u}$ are
equimeasurable
and, by (ii), every level set of $\widehat{u}$ has measure zero,
hence the same is true for $u$:
therefore, for any edge $e$ of $\G$ not in $\Gamma$,
the only possibility is that
$u\equiv 0$ on $e$ (for the moment, this cannot be excluded: as
$\mathop{\rm meas}(\{\widehat{u}>0\})=+\infty$,
the quantity
$\mathop{\rm meas}(\{u=0\})$ cannot be computed by complementation).
However, since $u\equiv 0$ outside $\Gamma$, considering $\Gamma$ as
a subgraph, the restriction $u_{\vert\Gamma}$ satisfies
\[
u_{\vert\Gamma}\in H^1_\mu(\Gamma),
\quad
E(u_{\vert\Gamma},\Gamma)=E(u,\G)=E(\widehat{u},\R).
\]
Since the
$\gamma_i:[0,+\infty)\to \Gamma$ are injective and parameterized by arclength,
it follows from \eqref{defv} and the previous relations that
\[
v\in H^1_\mu(\R),\quad
E(v,\R)=E(u_{\vert\Gamma},\Gamma)=E(\widehat{u},\R).
\]
Then, by (ii), $v$ is necessarily the translate of
a soliton of mass $\mu$ but, since $v(0)=M=\widehat{u}(0)$,
$v$ is centered at the origin and so $v =\widehat{u}=\phi_\mu$. In particular, $v>0$ whence also
$u_{\vert\Gamma}>0$: then, by continuity, no edge of $\G$ where $u\equiv 0$ can be attached to
$\Gamma$, hence (iii) is proved since $\G$ is connected.

Now, as $v=\phi_\mu$, $v$ is injective if restricted to either $\R^+$ or $\R^-$, hence
\[
\forall z_1,z_2\geq 0,\quad
\gamma_i(z_1)=\gamma_i(z_2)\quad\Rightarrow\quad
v(z_1)=v(z_2)
\quad\Rightarrow\quad
z_1=z_2,
\]
which shows that each $\gamma_i$ is a \emph{simple} curve in $\G$.
Recall that, by (H$'$), the images of the $\gamma_i$'s intersect at
at their starting point
$x_0$ and, possibly, also at finitely many other points $x_1,\ldots,x_n$ of $\G$.
If $n=0$ (i.e. if $x_0$ is the only intersection), then clearly $\Gamma$
is isometric to the real line $\R$, and by (iii) we see that $\G$ is necessarily
the graph of Example~\ref{example1} (a).

If $n=1$, then $\Gamma$ is a straight line with two points glued together,
hence $\G$ is necessarily
the graph of Example~\ref{example1} (b).

Finally, if $n>1$,
since $v=\varphi_\mu$ is an even and strictly radially decreasing function, we have
\begin{equation*}
\forall z_1,z_2>0,\quad
\gamma_1(z_1)=\gamma_2(z_2)=x_j\quad\Rightarrow\quad
v(z_1)=v(-z_2)
\quad\Rightarrow\quad
z_1=z_2.
\end{equation*}
Since the $\gamma_i$'s are parameterized by arclength, this shows that
for every $j\in\{1,\ldots,n\}$ the distance $a_j$ of $x_j$ from $x_0$, measured
along $\gamma_1$, is the same as the distance
measured along $\gamma_2$. And this forces $\G$
to be isometric to the graph in  Example~\ref{example1} (c).
\end{proof}

\section{Proof of the existence results} \label{sec:baffo}

Throughout this section, $\G$ is the graph described in Theorem~\ref{teoremdue}
(see Figure~\ref{FigT}). In dealing with \eqref{minE2}, however,
it is convenient to let $u=(\phi,\psi)$,
and
identify each $u\in H^1(\G)$ with a pair of functions $\phi,\psi$,
with $\phi\in H^1(\R)$, $\psi\in H^1(I)$ with $I=[0,\ell]$, satisfying the continuity
condition $\phi(0)=\psi(0)$: here, of course, the interval $I=(0,\ell)$ represents
the pendant of $\G$, while $\R$ represents the union of its two half-lines.
Then \eqref{minE2} is equivalent to the minimization problem
\begin{equation*}
\min\Bigl( E(\phi,\R)+E(\psi,I)\Bigr),\qquad
\phi\in H^1(\R),\quad \psi\in H^1(I)
\end{equation*}
subject to the constraints
\begin{equation}
\label{newc}
\phi(0)=\psi(0),\quad
\int_{-\infty}^\infty |\phi(x)|^2\,dx
+\int_0^\ell |\psi(x)|^2\,dx
=\mu.
\end{equation}

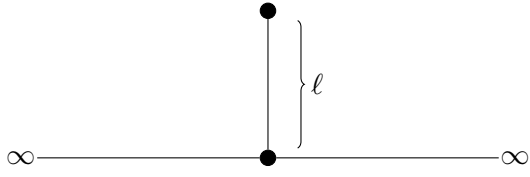
\begin{figure}[t]
\begin{center}
\begin{tikzpicture}
[scale=1.3,style={circle,
              inner sep=0pt,minimum size=7mm}]
\node at (-2.5,0) [infinito]  (L) {$\infty$};
\node at (0,0) [nodo] (O) {};
\node at (2.5,0) [infinito]  (R) {$\infty$};
%
\node at (0,1.5) [nodo] (V) {};

\draw[-] (L) -- (O);
\draw[-] (O) -- (R);
\draw [-] (O) -- (V) ;
\draw (0.5,0.75) node {$\ell$};
\draw [decorate,decoration=brace] (0.3,1.4) -- (0.3,0.1) ;

 %
\end{tikzpicture}
\end{center}
\caption{A straight line with a pendant of length $\ell>0$ attached to it.}
\label{FigT}
\end{figure}
\noindent

As mentioned in the Introduction,
the key
 to the proofs of Theorems \ref{teoremdue} and \ref{teoremdueB} is the
following lemma, that shows how the value of $E(u,\G)$ can be reduced,
for any $u\in H^1_\mu(\G)$, by a proper \emph{hybrid rearrangement} operation,
tailored to the topology of $\G$.

\newcommand\uorig{u}
\newcommand\ufinale{\widetilde u}
\newcommand\origphi{\phi}
\newcommand\origpsi{\psi}
\newcommand\finalphi{\widetilde\phi}
\newcommand\finalpsi{\widetilde\psi}

\begin{lemma}[hybrid rearrangement]\label{lemmaH} Assume that $\uorig\in H^1_\mu(\G)$,
with $\uorig> 0$ and $\mathop{\rm meas}(\{\uorig=t\})=0$ for every $t>0$.
Then
there exists
$\ufinale\in H^1_\mu(\G)$, $\ufinale=(\finalphi,\finalpsi)$,  with the following properties:
\begin{itemize}
\item[(i)] $\finalphi:\R\to\R$ is even and radially decreasing;
\item[(ii)] $\finalpsi:I\to\R$ is increasing, so that $\min \finalpsi=\finalpsi(0)=\finalphi(0)=\max \finalphi$;
\item[(iii)] $E(\ufinale,\G)\leq E(\uorig,\G)$ and, if equality occurs,
then
letting $\uorig=(\origphi,\origpsi)$ we necessarily have that  $\origpsi$ is increasing on
$[0,\ell]$ and $\min \origpsi=\max \origphi$.
\end{itemize}
\end{lemma}
\begin{proof}
By the assumptions on $\uorig$, it is easy to see that there exists $\tau>0$
such that
\begin{equation}
\label{chooset}
\mathop{\rm meas}(\{\uorig>\tau\})=\ell.
\end{equation}
The idea of the proof is to rearrange the portion of $\uorig$ above $\tau$ increasingly on
$I$, and the portion below $\tau$ symmetrically on $\R$. More precisely, let
\[
f=\uorig\wedge\tau,\qquad
g=(\uorig-\tau)^+,
\]
and observe that $f,g\in H^1(\G)$. Then one can define
$\widehat f$ and $g^*$,
the symmetric
and the decreasing rearrangements of $f$ and $g$, respectively, in such a way that
$\widehat f\in H^1(\R)$ and $g^*\in H^1(\R^+)$ by Remark~\ref{remomega}.
Note that: (i) since $\mathop{\rm meas}(\{g>0\})=\ell$ by \eqref{chooset}, $g^*$ is supported in
$[0,\ell]$; (ii) since $\mathop{\rm meas}(\{f=\tau\})=\ell$ by \eqref{chooset}, $f^*\equiv \tau$
on $[-\ell/2,\ell/2]$. Therefore, if we let
\[
\finalpsi(x)=\tau+g^*(\ell-x)\quad\forall x\in [0,\ell],
\quad\text{and}\quad
\finalphi(x)=
\begin{cases}
\widehat f(x+\ell/2) & \text{if $x\geq 0$,}\\
\widehat f(x-\ell/2) & \text{if $x< 0$,}
\end{cases}
\]
we have $\finalpsi(0)=\tau+g^*(\ell)=\tau$ and $\finalphi(0)=\widehat{f}(\ell/2)=\tau$, and the
first condition in \eqref{newc} is satisfied. Moreover,
as $\finalphi\in H^1(\R)$ and $\finalpsi\in H^1(I)$, letting $\ufinale=(\finalphi,\finalpsi)$ we have that $\ufinale\in H^1(\G)$.
In fact, since by construction for $t>0$
\begin{align*}
&\mathop{\rm meas}(\{\ufinale>t\})=
\mathop{\rm meas}(\{\finalphi>t\})
+
\mathop{\rm meas}(\{\finalpsi>t\})=\\
&\mathop{\rm meas}\bigl(\{\widehat f>t\}\setminus [-\ell/2,\ell/2]\bigr)
+
\mathop{\rm meas}\bigl(\{\tau+g^*>t\}\cap [0,\ell]\bigr)=\\
&\mathop{\rm meas}\bigl(\{\tau> \uorig\wedge \tau >t\}\bigr)
+
\mathop{\rm meas}\bigl(\{(\uorig-\tau)^+>(t-\tau)^+\}\bigr)
=\mathop{\rm meas}(\{\uorig>t\}),
\end{align*}
$\ufinale$ and $\uorig$ are equimeasurable and therefore also $\ufinale\in H^1_\mu(\G)$
(i.e., \eqref{newc} is fully satisfied). For the same reason,
\begin{equation}
\label{dis0}
\| \ufinale\|_{L^p(\G)}^p=
\int_{-\infty}^\infty |\finalphi(x)|^p\,dx
+\int_0^\ell |\finalpsi(x)|^p\,dx
=
\| \uorig\|_{L^p(\G)}^p.
\end{equation}
Now we claim that
\begin{equation}
\label{dis1}
\int_{-\infty}^\infty |\finalphi'(x)|^2\,dx
=
\int_{-\infty}^\infty |\widehat{f}'(x)|^2\,dx
\leq
\int_{\G} |(\uorig\wedge \tau)'|^2\,dx.
\end{equation}
The equality is immediate from the definition of $\finalphi(x)$,
since $\widehat f\equiv \tau$ on $[-\ell/2,\ell/2]$, while the inequality follows
from \eqref{eq:ens2}, as soon as we prove that
\begin{equation}
\label{eq:defN2}
N(t):=\#\{x\in \G\,:\,\, \uorig(x)\wedge \tau=t\}\geq 2
\quad\text{for a.e. $t\in (0,\tau)$.}
\end{equation}
Indeed, by \eqref{chooset}, the closed set $\{\uorig\geq\tau\}$ cannot be
\emph{strictly} contained inside the pendant edge of $\G$, since the latter
has measure $\ell$: therefore, we have that $\uorig(x)\geq\tau$ for at least
one point $x_0$ in the union of the two half-lines of $\G$ (possibly,
only at the triple junction). But this makes \eqref{eq:defN2} obvious,
even for \emph{every} $t\in (0,\tau)$,
since $\uorig\to 0$
as one approaches the two points at infinity of $\G$, both ways from $x_0$.

Similarly, since $g^*$ is supported in $[0,\ell]$ and $\finalpsi'(x)=(g^*)'(\ell-x)$,
\begin{equation}
\label{dis2}
\int_{0}^\ell |\finalpsi'(x)|^2\,dx
=
\int_{0}^\infty |(g^*)'(x) |^2 \,dx
\leq
\int_{\G} \left|\bigl((\uorig-\tau)^+\bigr)'\right|^2\,dx
\end{equation}
where now \eqref{eq:ens} is used. Taking the sum with \eqref{dis1}, we find
\[
\int_{\G} |\ufinale'(x)|^2\,dx\leq
\int_{\G} |\uorig'(x)|^2\,dx,
\]
which combined with \eqref{dis0} shows that $E(\ufinale,\G)\leq E(\uorig,\G)$.
Finally, if equality holds, then it must hold in \eqref{dis2} as well: from
the discussion after \eqref{eq:ens}, then, we see that the function $(\uorig-\tau)^+$
must have exactly \emph{one preimage} $x_t\in \G$, for almost every value
$t>0$ that it achieves on $\G$. Now split $\uorig=(\origphi,\origpsi)$, with
$\origphi\in H^1(\R)$ and $\origpsi\in H^1(I)$: if $M=\max \origphi >\tau$ then,
as $\origphi(x)\to 0$ when $|x|\to\infty$,
for \emph{every} $t\in (\tau,M)$ we would have $\# \origphi^{-1}(t)\geq 2$, hence
$(\uorig-\tau)^+$ would achieve \emph{every} value $t\in (\tau,M)$ \emph{at least twice}
on $\G$, which would be a contradiction. We deduce that $\origphi\leq \tau$,
hence \eqref{chooset} implies that the set $\{\uorig>\tau\}$ is contained in the
pendant edge of $\G$: indeed, since the pendant itself has measure $\ell$, we
obtain that $\origpsi\geq \tau$ on $I$, so that $\origpsi\geq\origphi$ and
$(\uorig-\tau)^+$ is supported on the pendant edge of $\G$. Since
$(\uorig-\tau)^+$ has one preimage for almost every value, we deduce that $\origpsi$ is monotone,
and this completes the proof.
\end{proof}

\begin{proof}[Proof of Theorem \ref{teoremdue}] To prove \eqref{infminore},
consider the soliton $\phi_\mu$, and
define
\[
\origphi(x):=\phi_\mu(|x|+\ell/2)\quad\forall x\in\R,\qquad
\origpsi(x):=\phi_\mu(x-\ell/2)\quad\forall x\in [0,\ell].
\]
It is clear that $\origphi\in H^1(\R)$, $\origpsi\in H^1([0,\ell])$, and \eqref{newc} is
satisfied. Then the function $\uorig=(\origphi,\origpsi)$ belongs to $H^1_\mu(\G)$, and clearly
$E(\uorig,\G)=E(\phi_\mu,\R)$. Moreover, $\uorig>0$ and $\mathop{\rm meas}(\{\uorig=t\})=0$
for every $t$, so that we can apply Lemma~\ref{lemmaH} and
obtain a new function $\ufinale\in H^1_\mu(\G)$ as claimed there: from (iii),
we see that $E(\ufinale,\G)<E(\phi,\G)$, because our $\origpsi$ is \emph{not} increasing
in $[0,\ell]$. Then we have
\[
E(\ufinale,\G)<E(\phi,\G)=E(\phi_\mu,\R),
\]
and this proves \eqref{infminore} since $\phi_\mu$ achieves the minimum in \eqref{infminore}.

To prove that the infimum in \eqref{infminore} is attained, let $u_n=(\phi_n,\psi_n)$
be a minimizing sequence, i.e.
\begin{equation}
\label{minseq}
\lim_n E(u_n,\G)=
\lim_n \bigl(E(\phi_n,\R)+E(\psi_n,I)\bigr)=\inf_{v\in H^1_\mu(\G)} E(v,\G).
\end{equation}
As $E(u_n,\G)=E(|u_n|,\G)$, we may assume that $u_n\geq 0$ and
even more, by a density argument using e.g. piecewise linear functions,
that each $u_n$ satisfies the assumptions
of Lemma~\ref{lemmaH}. Thus,
we may assume that $\phi_n,\psi_n$
satisfy properties as in (i) and (ii) therein, so that in particular
\begin{equation}
\label{unraddec}
\text{$\phi_n$ is even and radially decreasing.}
\end{equation}
Moreover from \eqref{coerc} we see that $\{u_n\}$ is bounded in $H^1(\G)$, hence
there exists a $u\in H^1(\G)$, $u=(\phi,\psi)$, such that up to subsequences
\begin{equation}
\psi_n\rightharpoonup \psi\quad\text{weakly in $H^1(I)$},\qquad
\phi_n\rightharpoonup \phi\quad\text{weakly in $H^1(\R)$.}
\end{equation}
Since $I=[0,\ell]$ is bounded, $\psi_n\to \psi$ strongly in $L^p(I)$, so that
\begin{equation}
\label{Evn}
E(\psi,I)\leq \liminf_n E(\psi_n,I)
\end{equation}
because the positive term in \eqref{NLSe} is lower semicontinuous.
Concerning $\phi_n$, clearly $\phi_n\to \phi$ uniformly on compact sets:
however, combining \eqref{unraddec} with the fact that $\phi(x)\to 0$
when $|x|\to\infty$, one can easily check that
$\phi_n\to \phi$ in $L^\infty(\R)$. Since $\{\phi_n\}$ is bounded in $L^2(\R)$
and $p>2$,
by interpolation we infer that also $\phi_n\to \phi$ strongly in $L^p(\R)$,
so that
\begin{equation*}
E(\phi,\R)\leq \liminf_n E(\phi_n,\R).
\end{equation*}
Combining with \eqref{Evn} and \eqref{minseq}, we have that
\begin{equation}
\label{sotto}
E(u,\G)=
E(\phi,\R)+E(\psi,I)\leq \inf_{v\in H^1_\mu(\G)} E(v,\G)<0
\end{equation}
(the last inequality follows from \eqref{infminore} and \eqref{ensol}).
Observe that, since $u_n\in H^1_\mu(\G)$, one has $\| u\|_{L^2(\G)}^2\leq\mu$
by semicontinuity: if the inequality were strict, one would have
$\sigma u\in H^1_\mu$ for a proper constant $\sigma>1$ (observe
$u\not\equiv 0$, since $E(u,\G)<0$ by \eqref{sotto}). But then, since
$\sigma^p>\sigma^2$ and $E(u,\G)<0$, we would find
\[
E(\sigma u,\G)=
\frac {\sigma^2} 2 \| u' \|^2_{L^2 (\G)}
- \frac {\sigma^p} p  \| u \|^p_{L^p (\G)}
<
\sigma^2 E(u,\G)
<
E(u,\G)
\]
which combined with \eqref{sotto} would be a contradiction, since now
$\sigma u\in H^1_\mu(\G)$. This shows that, $\| u\|_{L^2(\G)}^2=\mu$,
so $u\in H^1_\mu(\G)$ and $u$ is a minimizer.
\end{proof}

\begin{proof}[Proof of Theorem~\ref{teoremdueB}, (i)-(ii)] Let $\uorig\in H^1_\mu(\G)$ be a minimizer.
From (iii) in Proposition~\ref{necess}, we may assume that $\uorig>0$ on $\G$. Moreover, splitting
$\uorig=(\origphi,\origpsi)$ and letting $a=\origphi(0)$, $m=\| \origphi\|_{L^2(\R)}^2$, we see that $\origphi$
necessarily solves the double-constrained problem \eqref{pretta}, and therefore $\origphi$
is one of the functions described in Theorem~\ref{thretta} (i)--(iii): in any
case, one has $\mathop{\rm meas}(\{\origphi=t\})=0$ for every $t>0$. And the same is true
for $\origpsi$: indeed, as $\origpsi$ solves an ODE like \eqref{eulero} in $(0,\ell)$, all its level
sets have measure zero unless $\origpsi$ is constant on $[0,\ell]$. But if $\origpsi$ were constant
then $\origpsi(0)=\origpsi(\ell)$, and the function
\[
\phi:\R\to\R,\qquad
\phi(x)=\begin{cases}
\origphi(x) &\text{if $x<0$,}\\
\origpsi(x) &\text{if $0\leq x\leq \ell$,}\\
\origphi(x-\ell) &\text{if $x>\ell$,}
\end{cases}
\]
would therefore belong to $H^1_\mu(\R)$, with $E(\phi,\R)=E(\uorig,\G)$. But this would contradict
\eqref{infminore}, since $\uorig$ achieves the infimum there. This shows that $\uorig$ satisfies
the assumptions of Lemma~\ref{lemmaH}: since $E(\uorig,\G)$ cannot be diminished within $H^1_\mu(\G)$,
an \emph{equality} must occur in (iii) of Lemma~\ref{lemmaH}, and this proves that
$\origpsi$ is increasing on $[0,\ell]$ and $\min \origpsi=\max \origphi$. Since $\origphi(0)=\origpsi(0)$, we see
that $\origphi(0)=\max \origphi$, and this rules out the possibility that
$\origphi$ is as in (i) of Theorem~\ref{thretta}.
On the other hand, if $\origphi$ were as in (ii), i.e. $\origphi(x)=\phi_m(x)$,
then the Kirchhoff condition
\eqref{kir} at the triple junction of $\G$, namely
\[
-\origphi'(0^-)+\origphi'(0^+)+\origpsi'(0)=0,
\]
would reduce to $\origpsi'(0)=0$, since $\origphi=\phi_m$ would be of class $C^1(\R)$ with
$\origphi'(0)=0$. But then $\origphi,\origpsi$, other than solving the same ODE \eqref{eulero},
would also
satisfy the same initial conditions $\origphi(0)=\origpsi(0)$, $\origphi'(0)=\origpsi'(0)$: by uniqueness,
$\origpsi$ would then coincide with the restriction of $\origphi=\phi_m$ to $[0,\ell]$, and
since $\phi_m'(\ell)<0$, this would violate the Kirchhoff condition \eqref{kir}
at the tip of the pendant, namely $\origpsi'(\ell)=0$. Therefore, the only possibility is
that $\origphi$ is as in (iii) of Theorem~\ref{thretta}, and this completes the proof.
\end{proof}

\begin{proof}[Proof of Theorem~\ref{teoremdueB}, (iii)]
Given any two different lengths $\ell>\ell'>0$ for the pendant,
let $\G,\G'$ denote the corresponding graphs, and let $\uorig\in H^1_\mu(\G')$, $\uorig>0$, be
a minimizer on $\G'$: proceeding
\emph{exactly} as in the proof of Lemma~\ref{lemmaH},
we shall construct a function $\ufinale\in H^1_\mu(\G)$
such that $E(\uorig,\G')> E(\ufinale,\G)$.
Since $u$ has qualitative properties as in (i) and (ii)
of Theorem~\ref{teoremdueB},
we can certainly find $\tau=\tau(\ell)>0$ satisfying \eqref{chooset}
(now
$\uorig\in H^1_\mu(\G')$ and
$\ell$ is \emph{unrelated to} $\G'$,
but this is irrelevant
to the computations following \eqref{chooset}). Then we proceed
\emph{verbatim} with the construction of
$\ufinale\in H^1_\mu(\G)$ satisfying \eqref{dis0} and \eqref{dis2},
with the proviso that
$\G$ be replaced with  $\G'$ wherever $u$ is involved (the same applies
to
\eqref{eq:defN2} and \eqref{dis1}, hereafter).
 Since the structure of $u$ is as in (i)-(ii)
of Theorem~\ref{teoremdueB}, and the pendant of $\G'$ has length $\ell'$, from $\ell>\ell'$
and \eqref{chooset} we see that $f=u\wedge\tau$ has exactly two preimages for every
value $t\in (0,\tau)$ so that also \eqref{eq:defN2} and hence \eqref{dis1} are satisfied,
and  the same is true for $g=(u-\tau)^+$ for all those values $t>0$
\emph{small enough}, so that the inequality in \eqref{dis2}
is now \emph{strict}.
Therefore, we have
$E(\uorig,\G')> E(\ufinale,\G)$
hence the infimum in \eqref{infminore}, as a function of $\ell$,
 is strictly decreasing.
\end{proof}

\end{document}